\documentclass[oneside]{amsart}
\usepackage[T1]{fontenc}
\usepackage[latin9]{inputenc}
\usepackage{amsthm}
\usepackage{amssymb}
\usepackage{amstext}
\usepackage[numbers]{natbib}

	\numberwithin{equation}{section}
	\numberwithin{figure}{section}
	\theoremstyle{plain}
	\newtheorem{thm}{Theorem}[section]
	\theoremstyle{plain}
	\newtheorem{lem}[thm]{Lemma}
	\theoremstyle{definition}
	\newtheorem{defn}[thm]{Definition}
	\theoremstyle{remark}
	\newtheorem{rem}[thm]{Remark}
	\theoremstyle{plain}
	\newtheorem{prop}[thm]{Proposition}
	\theoremstyle{plain}
	\newtheorem{cor}[thm]{Corollary}

	\allowdisplaybreaks

\begin{document}
	\title[\hfill\protect\parbox{1\linewidth}{\center{A strong open mapping theorem for surjections\\ from cones onto Banach spaces}}]
	{A strong open mapping theorem for surjections from cones onto Banach spaces} 
	
	\begin{abstract}
		We show that a continuous additive positively homogeneous map from
a closed not necessarily proper cone in a Banach space onto a Banach
space is an open map precisely when it is surjective. This generalization
of the usual Open Mapping Theorem for Banach spaces is then combined
with Michael's Selection Theorem to yield the existence of a continuous
bounded positively homogeneous right inverse of such a surjective
map; a strong version of the usual Open Mapping Theorem is then
a special case. As another consequence, an improved version of the
analogue of And\^o's Theorem for an ordered Banach space is obtained
for a Banach space that is, more generally than in And\^o's Theorem,
a sum of possibly uncountably many closed not necessarily proper cones.
Applications are given for a (pre)-ordered Banach space and for various
spaces of continuous functions taking values in such a Banach space
or, more generally, taking values in an arbitrary Banach space that
is a finite sum of closed not necessarily proper cones.
	\end{abstract}

	% % % % % % % % % % % % % % % % %
	%	AUTHOR BIOGRAPHICAL DATA (name, address, email)
	% % % % % % % % % % % % % % % % %
	%Messerschmidt

\newcommand{\messerschmidtName}{Miek Messerschmidt}

\newcommand{\messerschmidtAddressLeiden}{
	Mathematical Institute, 
	Leiden University, 
	P.O. Box 9512, 
	\mbox{2300 RA} Leiden, 
	The Netherlands
}
\newcommand{\messerschmidtAddressNWU}{
	Unit for BMI,
	North-west University,
	Private Bag X6001,
	Potchefstroom,
	South Africa,
	2520
}

\newcommand{\messerschmidtAddress}{\messerschmidtAddressNWU}
\newcommand{\messerschmidtEmail}{mmesserschmidt@gmail.com}

%de Jeu

\newcommand{\dejeuName}{Marcel de Jeu}

\newcommand{\dejeuAddress}{
	Mathematical Institute, 
	Leiden University, 
	P.O. Box 9512, 
	\mbox{2300 RA} Leiden, 
	The Netherlands
}

\newcommand{\dejeuEmail}{mdejeu@math.leidenuniv.nl}

	%Marcel de jeu
	\author{\dejeuName}
	\address{\dejeuName, \dejeuAddress}
	\email{\dejeuEmail}
	
	%Miek Messerschmidt
	\author{\messerschmidtName}
	\address{\messerschmidtName, \messerschmidtAddressLeiden}
	\address{\footnote{Current address}\messerschmidtName, \messerschmidtAddress}
	\email{\messerschmidtEmail}

	% % % % % % % % % % % % % % % % %
	%	META INFO (KEYWORDS, SUBJECT CLASS)
	% % % % % % % % % % % % % % % % %
	\newcommand{\subjectClassesForThisPaper}{
		Primary: 47A05; 
		Secondary: 
			46A30\subjclasssep
			46B20\subjclasssep
			46B40\ignorespaces
}

\newcommand{\keywordsForThisPaper}{
	cone\keywordssep 
	Open Mapping Theorem\keywordssep 
	Banach space\keywordssep 
	bounded continuous right inverse\keywordssep 
	ordered Banach space\ignorespaces
}

	%Seperators for keywords and subject classes "\sep" or ", "
	\newcommand{\keywordssep}{, }
	\newcommand{\subjclasssep}{, }
	
	\subjclass[2010]{\subjectClassesForThisPaper}
	
	\maketitle
	
	% % % % % % % % % % % % % % % % %
	%	Content
	% % % % % % % % % % % % % % % % %
	
	\section{Introduction}

Consider the following question, that arose in other research of the 
authors: Let $X$ be a real Banach space, ordered by a closed generating
proper cone $X^{+}$, and let $\Omega$ be a topological space. Then
the Banach space $C_{0}(\Omega,X)$, consisting of the continuous
$X$-valued functions on $\Omega$ vanishing at infinity, is ordered
by the natural closed proper cone $C_{0}(\Omega,X^{+})$. Is this
cone also generating? If $X$ is a Banach lattice, then the answer
is affirmative. Indeed, if $f\in C_{0}(\Omega,X)$, then $f=f^{+}-f^{-}$,
where $f^{\pm}(\omega):=f(\omega)^{\pm}\,\,(\omega\in\Omega)$. Since
the maps $x\mapsto x^{\pm}$ are continuous, $f^{\pm}$ is continuous,
and since $\Vert f^{\pm}(\omega)\Vert\leq\Vert f(\omega)\Vert\,\,(\omega\in\Omega)$,
$f^{\pm}$ vanishes at infinity. Thus a decomposition as desired has
been obtained. For general $X$, the situation is not so clear. The
natural approach is to consider a pointwise decomposition as in the
Banach lattice case, but for this to work we need to know that at
the level of $X$ the constituents $x^{\pm}$ in a decomposition $x=x^{+}-x^{-}$
can be chosen in a continuous and simultaneously also bounded (in
an obvious terminology) fashion, as $x$ varies over $X$. Boundedness
is certainly attainable, due to the following classical result:
\begin{thm}
\label{thm:Ando_theorem}\textup{(}And\^o's Theorem \citep{Ando}\textup{)}
Let $X$ be a real Banach space ordered by a closed generating proper
cone $X^{+}\subseteq X$. Then there exists a constant $K>0$ with
the property that, for every $x\in X$, there exist $x^{\pm}\in X^{+}$
such that $x=x^{+}-x^{-}$ and $\Vert x^{\pm}\Vert\leq K\Vert x\Vert$.
\end{thm}
Continuity is not asserted, however. Hence we are not able to settle
our question in the affirmative via And\^o's Theorem alone and stronger
results are needed. With $\Omega$ a compact Hausdorff space, it is
a consequence of a result due to Asimow and Atkinson \citep[Theorem 2.3]{AsimowAtkinson}
that $C(\Omega,X^{+})$ is generating in $C(\Omega,X)$ when $X^{+}$
is closed and generating in $X$. A similar result due to Wickstead
\citep[Theorem 4.4]{Wickstead} establishes this for $C_{0}(\Omega,X)$
when $\Omega$ is locally compact (cf.\,Remark \ref{rem:asmiow-atkinson-wikstead}).
We will also retrieve these results, but by a different method, namely
by establishing the general existence of a continuous bounded decompositions,
analogous to that for Banach lattices, as a special case of Theorem
\ref{thm:continuous_functions_arbitrary_topological_space} below. 

In fact, although the above question and results are in the context
of ordered Banach spaces, it will become clear in this paper that
for these spaces one is merely looking at a particular instance of
more general phenomena. In short: if $T:C\to X$ is a continuous additive
positively homogeneous surjection from a closed not necessarily proper
cone in a Banach space onto a Banach space, then $T$ has a well-behaved
right inverse, and (stronger) versions of theorems such as And\^o's,
where several cones in one space are involved, are then almost immediately
clear. We will now elaborate on this, and at the same time explain
the structure of the proofs.

The usual notation to express that $X^{+}$ is generating is to write
$X=X^{+}-X^{+}$, but the actually relevant point turns out to be
that $X=X^{+}+(-X^{+})$ is a sum of two closed cones: the fact that
these are related by a minus sign is only a peculiarity of the context.
In fact, if $X$ is the sum of possibly uncountably many closed cones,
which need not be proper (this is redundant in And\^o's Theorem),
then it is possible to choose a bounded decomposition: this is the
content of the first part of Theorem \ref{thm:continuous_selections_from_generating_cones_ell_one_case}.
However, this is in itself a consequence of the following more fundamental
result, a special case of Theorem \ref{thm:open_mapping_theorem}.
Since a Banach space is a closed cone in itself, it generalizes the
usual Open Mapping Theorem for Banach spaces, which is used in the
proof.
\begin{thm}
\textup{\label{thm:ch1Open_Mapping_Theorem_introduction_version}(}Open
Mapping Theorem\textup{)} Let $C$ be a closed cone in a real or complex
Banach space, not necessarily proper. Let $X$ be a real or complex
Banach space, not necessarily over the same field as the surrounding
space of $C$, and $T:C\to X$ a continuous additive positively homogeneous
map. Then the following are equivalent:
\begin{enumerate}
\item $T$ is surjective;  
\item There exists some constant $K>0$ such that, for every $x\in X$,
there exists some $c\in C$ with \textup{$x=Tc$} and $\|c\|\leq K\|x\|$;
\item $T$ is an open map; 
\item 0 is an interior point of $T(C)$.
\end{enumerate}
\end{thm}
As an illustration of how this can be applied, suppose $X=\sum_{i\in I}C_{i}$
is the sum of a finite (for the ease of formulation) number of closed
not necessarily proper cones. We let $Y$ be the sum of $|I|$ copies
of $X,$ and let $C\subset Y$ be the direct sum of the $C_{i}$'s.
Then the natural summing map $T:C\to X$ is surjective by assumption,
so that part (2) of Theorem \ref{thm:ch1Open_Mapping_Theorem_introduction_version}
provides a bounded decomposition. And\^o's Theorem corresponds to
the case where $X$ is the image of $X^{+}\times(-X^{+})\subset X\times X$
under the summing map. 

In this fashion, generalizations of And\^o's Theorem are obtained
as a consequence of an Open Mapping Theorem. However, this still does
not resolve the issue of a decomposition that is not only bounded,
but continuous as well. A possible attempt to obtain this would be
the following: if $T:Y\to X$ is a continuous linear surjection between
Banach spaces (or even Fr\'echet spaces), then $T$ has a continuous
right inverse, see \citep[Corollary 17.67]{AliprantisBorder}. The proof
is based on Michael's Selection Theorem, which we will recall in Section
2. Conceivably, the proof as in \citep{AliprantisBorder} could be modified
to yield a similar statement for a continuous surjective additive
positively homogeneous $T:C\to X$ from a closed cone $C$ in a Banach
space onto $X$. In that case, if $X=\sum_{i\in I}C_{i}$ is a finite
(say) sum of closed not necessarily proper cones, the setup with product
cone and summing map would yield the existence of a continuous decomposition,
but unfortunately this time there is no guarantee for boundedness.
Somehow the generalized Open Mapping Theorem as in Theorem \ref{thm:ch1Open_Mapping_Theorem_introduction_version}
and Michael's Selection Theorem must be combined. The solution lies
in a refinement of the correspondences to which Michael's Selection
Theorem is to be applied, and take certain subadditive maps on $C$
into account from the very beginning. In the end, one of these maps
will be taken to be the norm on $C$, and this provides the desired
link between the generalized Open Mapping Theorem and Michael's Selection
Theorem, cf.\  the proof of Proposition \ref{prop:key_proposition}.
It is along these lines that the following is obtained. It is a special
case of Theorem \ref{thm:continuous_right_inverses_exists} and, as
may be clear by now, it implies the existence of a continuous bounded
(and even positively homogeneous) decomposition if $X=\sum_{i\in I}C_{i}$.
It also shows that, if $T:Y\to X$ is a continuous linear surjection
between Banach spaces, then it is not only possible to choose a bounded
right inverse for $T$ (a statement equivalent to the usual Open Mapping
Theorem), but also to choose a bounded right inverse that is, in addition,
continuous and positively homogeneous. 
\begin{thm}
\label{thm:bounded_continuous_inverses_exist_introduction_version} 
\textup{(}Strong Open Mapping Theorem\textup{)}
Let
$X$ and $Y$ be real or complex Banach spaces, not necessarily over
the same field, and let $C$ be a closed not necessarily proper cone
in $Y$. Let $T:C\to X$ be a surjective continuous additive positively
homogeneous map.

Then there exists a constant $K>0$ and a continuous positively homogeneous
map $\gamma:X\to C$, such that:
\begin{enumerate}
\item $T\circ\gamma=\mbox{id}_{X}$;
\item $\Vert\gamma(x)\Vert\leq K\Vert x\Vert$, for all $x\in X$.
\end{enumerate}
\end{thm}
The underlying Proposition \ref{prop:key_proposition} is the core
of this paper. It is reworked into the somewhat more practical Theorems
\ref{thm:continuous_right_inverses_exists} and \ref{thm:improving_a_family_of_right_inverses_is_possible},
but this is all routine, as are the applications in Section \ref{sec:applications}.
For example, the following result (Corollary \ref{cor:ando_improved})
is virtually immediate from Section \ref{sec:main_results}. We cite
it in full, not only because it shows concretely how And\^o's Theorem
figuring so prominently in our discussion so far can be strengthened,
but also to enable us to comment on the interpretation of the various
parts of this result and similar ones.
\begin{thm}
Let $X$ be a real \textup{(}pre\textup{)}-ordered Banach space, \textup{(}pre\textup{)}-ordered
by a closed generating not necessarily proper cone $X^{+}$. Let $J$
be a finite set, possibly empty, and, for all $j\in J$, let $\rho_{j}:X\times X\to\mathbb{R}$
be a continuous seminorm or a continuous linear functional. Then:
\begin{enumerate}
\item There exist a constant $K>0$ and continuous positively homogeneous
maps $\gamma^{\pm}:X\to X^{+}$, such that:

\begin{enumerate}
\item $x=\gamma^{+}(x)-\gamma^{-}(x)$, for all $x\in X$;
\item $\Vert\gamma^{+}(x)\Vert+\Vert\gamma^{-}(x)\Vert\leq K\Vert x\Vert$,
for all $x\in X$.
\end{enumerate}
\item If $K>0$ and $\alpha_{j}\in\mathbb{R\,}(j\in J)$ are constants,
then the following are equivalent:

\begin{enumerate}
\item For every $\varepsilon>0$, there exist maps $\gamma_{\varepsilon}^{\pm}:S_{X}\to X^{+}$,
where $S_{X}:=\{x\in X:\|x\|=1\}$, such that:

\begin{enumerate}
\item $x=\gamma_{\varepsilon}^{+}(x)-\gamma_{\varepsilon}^{-}(x)$, for
all $x\in S_{X}$;
\item $\Vert\gamma_{\varepsilon}^{+}(x)\Vert+\Vert\gamma_{\varepsilon}^{-}(x)\Vert\leq(K+\varepsilon)$,
for all $x\in S_{X}$;
\item $\rho_{j}((\gamma_{\varepsilon}^{+}(x),\gamma_{\varepsilon}^{-}(x))\leq(\alpha_{j}+\varepsilon)$,
for all $x\in S_{X}$ and $j\in J$.
\end{enumerate}
\item For every $\varepsilon>0$, there exist continuous positively homogeneous
maps $\gamma_{\varepsilon}^{\pm}:X\to X^{+}$, such that:

\begin{enumerate}
\item $x=\gamma_{\varepsilon}^{+}(x)-\gamma_{\varepsilon}^{-}(x)$, for
all $x\in X$;
\item $\Vert\gamma_{\varepsilon}^{+}(x)\Vert+\Vert\gamma_{\varepsilon}^{-}(x)\Vert\leq(K+\varepsilon)\|x\|$,
for all $x\in X$;
\item $\rho_{j}((\gamma_{\varepsilon}^{+}(x),\gamma_{\varepsilon}^{-}(x))\leq(\alpha_{j}+\varepsilon)\|x\|$,
for all $X$ and $j\in J$.
\end{enumerate}
\end{enumerate}
\end{enumerate}
\end{thm}
The existence of a bounded continuous positively homogeneous decomposition
in part (1) is of course a direct consequence of Theorem \ref{thm:bounded_continuous_inverses_exist_introduction_version}.
Naturally, the argument as for Banach lattices then shows that $C_{0}(\Omega,X)=C_{0}(\Omega,X^{+})-C_{0}(\Omega,X^{+})$
for an arbitrary topological space $\Omega$, so that our original
question has been settled in the affirmative. 

The equivalence under (2) has the following consequence: If there
exist maps $\gamma^{\pm}:S_{X}\to X^{+}$, such that $x=\gamma^{+}(x)-\gamma^{-}(x)$,
$\|\gamma^{+}(x)\|+\|\gamma^{-}(x)\|\leq K$, and $\rho_{j}((\gamma^{+}(x),\gamma^{-}(x))\leq\alpha_{j}$,
for all $x\in S_{X}$ and $j\in J$, then certainly maps as under
(2)(a) exist (take $\gamma_{\varepsilon}^{\pm}=\gamma^{\pm}$, for
all $\varepsilon>0$), and hence a family of much better behaved global
versions exists as under (2)(b), at an arbitrarily small price in
the constants. 

The possibility to include the $\rho_{j}$'s in part (2) (with similar
occurrences in other results) is a bonus from the refinement of the
correspondences to which Michael's Selection Theorem is applied. For
several issues, such as our original question concerning $C_{0}(\Omega,X),$
it will be sufficient to use part (1) and conclude that a continuous
bounded decomposition exists. In this paper we also include some applications
of part (2) with non-empty $J$. Corollary \ref{cor:continous_approximate_conormality}
shows that approximate $\alpha$-conormality of a (pre)-ordered Banach
space is equivalent with continuous positively homogeneous approximate
$\alpha$-conormality, and Corollary \ref{cor:equivalence_approximate_normality_for_X_and_for_function_spaces}
shows that approximate $\alpha$-conormality of $X$ is inherited
by various spaces of continuous $X$-valued functions on a topological
space.

\medskip

We emphasize that, although Banach spaces that are a sum of cones,
and ordered Banach spaces in particular, have played a rather prominent
role in this introduction, the actual underlying results are those
in Section \ref{sec:main_results}, valid for a continuous additive
positively homogeneous surjection $T:C\to X$ from a closed not necessarily
proper cone $C$ in a Banach space onto a Banach space $X$. That
is the heart of the matter.

\medskip

This paper is organized as follows.

Section \ref{sec:Preliminaries} contains the basis terminology and
some preliminary elementary results. The terminology is recalled in
detail, in order to avoid a possible misunderstanding due to differing
conventions. 

In Section \ref{sec:main_results} the Open Mapping Theorem for Banach
spaces and Michael's Selection Theorem are used to investigate surjective
continuous additive positively homogeneous maps $T:C\to X$.

Section \ref{sec:applications} contains the applications, rather
easily derived from Section \ref{sec:main_results}. Banach spaces
that are a sum of closed not necessarily proper cones are approached
via the naturally associated closed cone in a Banach space direct
sum and the summing map. The results thus obtained are then in turn
applied to a (pre)-ordered Banach space $X$ and to various spaces
of continuous $X$-valued functions.

\section{Preliminaries\label{sec:Preliminaries}}

In this section we establish terminology, include a few elementary
results concerning metric cones for later use, and recall Michael's
Selection Theorem.

If $X$ is a normed space, then $S_{X}:=\{x\in X:\Vert x\Vert=1\}$
denotes its unit sphere.

\subsection{Subsets of vector spaces}

For the sake of completeness we recall that a non-empty subset $A$
of a real vector space $X$ is \textit{star-shaped with respect to
0} if $\lambda x\in A$, for all $x\in A$ and $0\leq\lambda\leq1$,
and that it is \textit{balanced }if $\lambda x\in A$, for all $x\in A$
and $-1\leq\lambda\leq1$. $A$ is \textit{absorbing in $X$} if,
for all $x\in X$, there exists $\lambda>0$ such that $x\in\lambda A$.
$A$ is \textit{symmetric }if $A=-A.$

The next rather elementary property will be used in the proof of Proposition
\ref{prop:equivalent_norm}.
\begin{lem}
\label{lem:intersection_star_shaped_absorbing}Let $X$ be a real
vector space and suppose $A,B\subseteq X$ are star-shaped with respect
to 0 and absorbing. Then $A\cap B$ is star-shaped with respect to
0 and absorbing.\end{lem}
\begin{proof}
It is clear that $A\cap B$ is star-shaped with respect to 0. Let
$x\in X$, then, since $A$ is absorbing, $x\in\lambda A$ for some
$\lambda>$0. The fact that $A$ is star-shaped with respect to 0
then implies that $x\in\lambda^{\prime}A$ for all $\lambda^{\prime}\geq\lambda$.
Likewise, $x\in\mu B$ for some $\mu>0$, and then $x\in\mu^{\prime}B$
for all $\mu^{\prime}\geq\mu$. Hence $x\in\max(\lambda,\mu)(A\cap B)$.
\end{proof}
A subset $C$ of the real or complex vector space $X$ is called a
\emph{cone in $X$ }if $C+C\subseteq C$ and $\lambda C\subseteq C$,
for all $\lambda\geq0$. We note that we do not require $C$ to be
a proper cone, i.e., that $C\cap(-C)=\{0\}$.

\subsection{Cones}

The cones figuring in the applications in Section \ref{sec:applications}
are cones in Banach spaces, but one of the two main results leading
to these applications, the Open Mapping Theorem (Theorem \ref{thm:open_mapping_theorem}),
can be established for the following more abstract objects.
\begin{defn}
\label{def:abstract_cone}Let $C$ be a set equipped with operations $+:C\times C\to C$
and $\cdot:\mathbb{R}_{\geq0}\times C\to C$. Then $C$ will be called
an \emph{abstract cone} if there exists an element $0\in C$, such
that the following hold for all $u,v,w\in C$ and $\lambda,\mu\in\mathbb{R}_{\geq0}$:
\begin{enumerate}
\item $u+0=u$;
\item $(u+v)+w=u+(w+v)$;
\item $u+v=v+u$;
\item $u+v=u+w$ implies $v=w$;
\item $1u=u$ and $0u=0$;
\item $(\lambda\mu)u=\lambda(\mu u$);
\item $(\lambda+\mu)u=\lambda u+\mu u$;
\item $\lambda(u+v)$=$\lambda u+\lambda v$.
\end{enumerate}
Here we have written $\lambda\cdot u$ as $\lambda u$ for short,
as usual. 

The natural class of maps between two cones $C_{1}$ and $C_{2}$
consists of the \textit{additive and positively homogeneous} maps,
i.e., the maps $T:C_{1}\to C_{2}$ such that $T(u+v)=Tu+Tv$ and $T(\lambda u)=\lambda u$,
for all $u,v\in C$ and $\lambda\geq0$. 
\end{defn}
{}
\begin{defn}
\label{def:metric-cone}A pair $(C,d)$ will be called a \emph{metric
cone} if $C$ is an abstract cone and $d:C\times C\to\mathbb{R}_{\geq0}$
is a metric, satisfying
\begin{eqnarray}
d(0,\lambda u) & = & \lambda d(0,u),\label{eq:positive_homogeneity}\\
d(u+v,u+w) & \leq & d(v,w),\label{eq:sub_translation_invariance}
\end{eqnarray}
for every $u,v,w\in C$ and $\lambda\geq0$ . A metric cone $(C,d)$
is a \textit{complete metric cone} if it is a complete metric space.\end{defn}
\begin{rem}
{}
\begin{enumerate}
\item \ Once Michael's Selection Theorem is combined with the Open Mapping
Theorem (Theorem \ref{thm:open_mapping_theorem}), $C$ will be a
closed not necessarily proper cone in a Banach space, and the metric
will be induced by the norm. In that case it is translation invariant,
but for the Open Mapping Theorem as such requiring (\ref{eq:sub_translation_invariance})
is already sufficient. The natural similar requirement $d(0,\lambda u)\leq\lambda d(0,u)$,
which is likewise sufficient for the proofs, is easily seen to be
actually equivalent to requiring equality as in (\ref{eq:positive_homogeneity})
above.
\item Although we will not use this, we note that, if $(C,d)$ is a metric
cone, then $+:C\times C\to C$ is easily seen to be continuous, as
is the map $\lambda\to\lambda u$ from $\mathbb{R}_{\geq0}$ into
$C$, for each $u\in C$.
\end{enumerate}
\end{rem}
The following elementary results will be needed in the proof of Proposition
\ref{prop:equivalent_norm}.
\begin{lem}
\label{lem:elementary_metric_cone_results}Let $(C,d)$ be a metric
cone as in Definition \ref{def:abstract_cone}.
\begin{enumerate}
\item If $c_{1},\ldots,c_{n}\in C$, then $d(0,\sum_{i=1}^{n}c_{i})\leq\sum_{i=1}^{n}d(0,c_{i}).$
\item Let $X$ be a real or complex normed space and suppose $T:C\to X$
is positively homogeneous and continuous at 0. Then $T$ maps metrically
bounded subsets of $C$ to norm bounded subsets of $X$
\end{enumerate}
\end{lem}
\begin{proof}
For the first part, using the triangle inequality and (\ref{eq:sub_translation_invariance})
we conclude that $d(0,\sum_{i=1}^{n}c_{i})\leq d(0,c_{n})+d(c_{n},\sum_{i=1}^{n}c_{i})\leq d(0,c_{n})+d(0,\sum_{i=1}^{n-1}c_{i})$,
so the statement follows by induction.

As to the second part, by continuity of $T$ at zero there exists
some $\delta>0$ such that $\|Tc\|<1$ holds for all $c\in C$ satisfying
$d(0,c)<\delta$. If $U\subseteq C$ is bounded, choose $r>0$ such
that $U\subseteq\{c\in C:d(0,c)<r\}$. Since $d(0,\lambda u)=\lambda d(0,u)$,
for all $u\in C$ and $\lambda\geq0$, $\delta r^{-1}U\subseteq\{c\in C:d(0,c)<\delta\}$.
Then by positive homogeneity of $T$, $\sup_{u\in U}\|Tu\|\leq\delta^{-1}r<\infty$.
\end{proof}

\subsection{Correspondences}

Our terminology and definitions concerning correspondences follow
that in \citep{AliprantisBorder}. Let $A,B$ be sets. A map $\varphi$
from $A$ into the power set of $B$ is called a \emph{correspondence}\textit{
from $A$ into $B$}, and is denoted by $\varphi:A\twoheadrightarrow B$.
A \emph{selector }for a correspondence $\varphi:A\twoheadrightarrow B$
is a function $\sigma:A\to B$ such that $\sigma(x)\in\varphi(x)$
for all $a\in A$. If $A$ and $B$ are topological spaces, we say
a correspondence $\varphi$ is \emph{lower hemicontinuous} if, for
every $a\in A$ and every open set $U\subseteq B$ with $\varphi(a)\cap U\neq\emptyset$,
there exists an open neighborhood $V$ of $a$ in $A$ such that $\varphi(a')\cap U\neq\emptyset$
for every $a'\in V$. The following result is the key to the proof
of Proposition \ref{prop:continuous_right_inverses_Frechet} concerning
the existence of continuous sections for surjections of cones onto
normed spaces.
\begin{thm}
\label{thm:selection_theorem}\textup{(}Michael's Selection Theorem
\citep[Theorem  17.66]{AliprantisBorder}\textup{)} Let $\varphi:A\twoheadrightarrow Y$
be a correspondence from a paracompact space A into a real or complex
Fr\'echet space $Y$. If $\varphi$ is lower hemicontinuous and has
non-empty closed convex values, then it admits a continuous selector.
\end{thm}

\section{Main results\label{sec:main_results}}

In this section we establish our main results, Theorems \ref{thm:open_mapping_theorem},
\ref{thm:continuous_right_inverses_exists} and \ref{thm:improving_a_family_of_right_inverses_is_possible}.
Theorem \ref{thm:open_mapping_theorem} is an Open Mapping Theorem
for surjections from complete metric cones onto Banach spaces; its
proof is based on the usual Open Mapping Theorem. Together with the
technical Proposition \ref{prop:continuous_right_inverses_Frechet}
(based on Michael's Selection Theorem) it yields the key Proposition
\ref{prop:key_proposition}. This is then reworked into two more practical
results. The first of these, Theorem \ref{thm:continuous_right_inverses_exists},
guarantees the existence of continuous bounded positively homogeneous
right inverses, while the second, Theorem \ref{thm:improving_a_family_of_right_inverses_is_possible},
shows that the existence of a family of possibly ill-behaved local
right inverses implies the existence of a family well-behaved global
ones.

As before, if $X$ is a normed space, then $S_{X}:=\{x\in X:\Vert x\Vert=1\}$
is its unit sphere.

We start with the core of the proof of the Open Mapping Theorem, which
employs a certain Minkowski functional. The use of such functionals
when dealing with cones and Banach spaces goes back to Klee \citep{Klee}
and And\^ o \citep{Ando}. 
\begin{prop}
\label{prop:equivalent_norm}Let $(C,d)$ be a complete metric cone
as in Definition \ref{def:abstract_cone}. Let $X$ be a real Banach
space and $T:C\to X$ a continuous additive positively homogeneous
surjection. Let $B:=\{c\in C:d(0,c)\leq1\}$ denote the closed unit
ball around zero in $C$, and define $V:=T(B)\cap(-T(B))$. Then $V$
is an absorbing convex balanced subset of $X$ and its Minkowski functional
$\|\cdot\|_{V}:X\to\mathbb{R}$, given by $\|x\|_{V}:=\inf\{\lambda>0:x\in\lambda V\}$,
for $x\in X$, is a norm on $X$ that is equivalent to the original
norm on $X$.\end{prop}
\begin{proof}
It follows from Lemma \ref{lem:elementary_metric_cone_results} and
Definition \ref{def:metric-cone} that $B:=\{c\in C:d(0,c)\leq1\}$
is convex. Hence $T(B)$ is convex and contains zero, since $T$ is
additive and positively homogeneous. Since $0\in T(B)$, its convexity
implies that $T(B)$ is star-shaped with respect to 0. Furthermore,
$T(B)$ is absorbing, as a consequence of the surjectivity and positive
homogeneity of $T$ and (\ref{eq:positive_homogeneity}). Thus $T(B)$
is star-shaped with respect to 0 and absorbing, and since this implies
the same properties for $-T(B)$, Lemma \ref{lem:intersection_star_shaped_absorbing}
shows that $V:=T(B)\cap(-T(B))$ is star-shaped with respect to 0
and absorbing. As $V$ is clearly symmetric, its star-shape with respect
to 0 implies that it is balanced. Furthermore, the convexity of $T(B)$
implies that $V$ is convex. All in all, $V$ is an absorbing convex
balanced subset of the real vector space $X$, and hence its Minkowski
functional $\|\cdot\|_{V}$ is a seminorm by \citep[Theorem 1.35]{Rudin}.
Because $T$ is continuous at 0, Lemma \ref{lem:elementary_metric_cone_results}
implies that $\sup_{y\in V}\|y\|\leq M$ for some $M>0$. If $x\in X$
and $\lambda>\Vert x\Vert_{V},$ then the definition of $\|\cdot\|_{V}$
and the star-shape of $V$ with respect to 0 imply that $x\in\lambda V$,
so that $\|x\|\leq\lambda M$ . Hence
\begin{equation}
\Vert x\Vert\leq M\Vert x\Vert_{V}\quad(x\in X).\label{eq:inequality_between_norms}
\end{equation}
We conclude that $\|\cdot\|_{V}$ is a norm on $X$. In view of (\ref{eq:inequality_between_norms}),
the equivalence of $\|\cdot\|_{V}$ and $\|\cdot\|$ is an immediate
consequence of the Bounded Inverse Theorem for Banach spaces, once
we know that $(X,\|\cdot\|_{V})$ is complete. We will now proceed
to show this, using the completeness of $(C,d)$.

To this end, it suffices to show $\|\cdot\|_{V}$-convergence of all
$\|\cdot\|_{V}$-absolutely convergent series. Let $\{x_{i}\}_{i=1}^{\infty}$
be a sequence in $X$ such that $\sum_{i=1}^{\infty}\|x_{i}\|_{V}<\infty$.
Since $\|x\|\leq M\|x\|_{V}$ for all $x\in X$, $\sum_{i=1}^{\infty}\|x_{i}\|<\infty$
also holds, hence by completeness of $X$ the $\|\cdot\|$-sum $x_{0}:=\sum_{i=1}^{\infty}x_{i}$
exists. We claim that $\left\Vert x_{0}-\sum_{i=1}^{N-1}x_{i}\right\Vert _{V}\to0$
as $N\to\infty$, i.e., that $x_{0}$ is also the $\|\cdot\|_{V}$-sum
of this series.

In order to establish this, we start by noting that, for $x\in X$,
there exists $x^{\prime}\in V$ such that $x=2\Vert x\Vert_{V}x^{\prime}.$
This is clear if $\Vert x\Vert_{V}=0$. If $\Vert x\Vert_{V}\neq0$,
then $2\Vert x\Vert_{V}>\Vert x\Vert_{V}$ and, as already observed
earlier, this implies that $x\in2\Vert x\Vert_{V}V$. Therefore for
every $i\in\mathbb{N}$ there exists $x_{i}^{\prime}\in V$ satisfying
$x_{i}=2\|x_{i}\|_{V}x_{i}^{\prime}$. Since $V\subset T(B)$, for
every $i\in\mathbb{N}$ there exists $b_{i}\in B$ such that $Tb_{i}=x_{i}^{\prime}$,
so that $x_{i}=2\Vert x_{i}\Vert_{V}Tb_{i}=T(2\Vert x_{i}\Vert_{V}b_{i})$.
Note that $d(0,2\|x_{i}\|_{V}b_{i})=2\|x_{i}\|_{V}d(0,b_{i})\leq2\|x_{i}\|_{V}$. 

From (\ref{eq:sub_translation_invariance} and Lemma \ref{lem:elementary_metric_cone_results})
it then follows that, for any fixed $N\in\mathbb{N}$ and all $n,m\in\mathbb{N}$
with $N\leq m\leq n$, 
\begin{eqnarray*}
d\left(\sum_{i=N}^{m}2\|x_{i}\|_{V}b_{i},\sum_{i=N}^{n}2\|x_{i}\|_{V}b_{i}\right) & \leq & d\left(0,\sum_{i=m+1}^{n}2\|x_{i}\|_{V}b_{i}\right)\\
 & \leq & \sum_{i=m+1}^{n}d(0,2\Vert x_{i}\Vert_{V}b_{i})\\
 & \leq & \sum_{i=m+1}^{n}2\|x_{i}\|_{V}.
\end{eqnarray*}
 We conclude that, for any fixed $N\in\mathbb{N}$, $\left\{ \sum_{i=N}^{n}2\|x_{i}\|_{V}b_{i}\right\} _{n=N}^{\infty}$
is a Cauchy sequence in $(C,d)$ and hence, by completeness, converges
to some $c_{N}\in C$. Using Lemma \ref{lem:elementary_metric_cone_results}
again we find that 
\[
d(0,c_{N})=\lim_{n\to\infty}d\left(0,\sum_{i=N}^{n}2\|x_{i}\|_{V}b_{i}\right)\leq\limsup_{n\to\infty}\sum_{i=N}^{n}d(0,2\Vert x_{i}\Vert_{V}b_{i})\leq\sum_{i=N}^{\infty}2\|x_{i}\|_{V},
\]
so that $c_{N}\in\left(\sum_{i=N}^{\infty}2\|x_{i}\|_{V}\right)B$,
as a consequence of (\ref{eq:positive_homogeneity}). By the continuity,
additivity and positive homogeneity of $T$, $Tc_{N}=\lim_{n\to\infty}T\left(\sum_{i=N}^{n}2\|x_{i}\|_{V}b_{i}\right)=\lim_{n\to\infty}\sum_{i=N}^{n}T(2\|x_{i}\|_{V}b_{i})=\lim_{n\to\infty}\sum_{i=N}^{n}x_{i}=\sum_{i=N}^{\infty}x_{i}$
with respect to the $\|\cdot\|$-topology. Hence 
\[
x_{0}-\sum_{i=1}^{N-1}x_{i}=\sum_{i=N}^{\infty}x_{i}=Tc_{N}\in\left(\sum_{i=N}^{\infty}2\|x_{i}\|_{V}\right)T(B).
\]
Similarly, the inclusion $V\subset-T(B)$ implies that, for every
$i\in\mathbb{N},$ there exists $\tilde{b}_{i}\in B$ such that $-T\tilde{b}_{i}=x_{i}^{\prime}$.
Then $\tilde{c}_{N}=\lim_{n\to\infty}\sum_{i=N}^{n}2\Vert x_{i}\Vert_{V}\tilde{b}_{i}$
exists for all $N\in\mathbb{N}$, $\tilde{c}_{N}\in\left(\sum_{i=N}^{\infty}2\|x_{i}\|_{V}\right)B$,
and $T\tilde{c}_{N}=-\sum_{i=N}^{\infty}x_{i}$, so that 
\[
x_{0}-\sum_{i=1}^{N-1}x_{i}=\sum_{i=N}^{\infty}x_{i}=-T\tilde{c}_{N}\in\left(\sum_{i=N}^{\infty}2\|x_{i}\|_{V}\right)(-T(B)).
\]
We conclude that $x_{0}-\sum_{i=1}^{N-1}x_{i}\in\left(\sum_{i=N}^{\infty}2\|x_{i}\|_{V}\right)V$.
Therefore, 
\[
\left\Vert x_{0}-\sum_{i=1}^{N-1}x_{i}\right\Vert _{V}\leq\sum_{i=N}^{\infty}2\|x_{i}\|_{V}\to0
\]
as $N\to\infty$, and hence $(X,\|\cdot\|_{V})$ is complete.
\end{proof}
The Open Mapping Theorem is now an easy consequence.
\begin{thm}
\label{thm:open_mapping_theorem}\textup{(}Open Mapping Theorem\textup{)}
Let $(C,d)$ be a complete metric cone as in Definition \ref{def:abstract_cone};
for example, $C$ could be a closed not necessarily proper cone in
a Banach space. Let $X$ be a real or complex Banach space and $T:C\to X$
a continuous additive positively homogeneous map. Then the following
are equivalent:
\begin{enumerate}
\item $T$ is surjective;
\item There exists some constant $K>0$ such that, for every $x\in X$,
there exists some $c\in C$ with \textup{$x=Tc$} and $d(0,c)\leq K\|x\|$;
\item $T$ is an open map;
\item 0 is an interior point of T(C).
\end{enumerate}
\end{thm}
\begin{proof}
Given the nature of the statements in (1)-(4), we may assume that
$X$ is a real Banach space, by viewing a complex one as such if necessary.
We first prove that (1) implies (2). By Proposition \ref{prop:equivalent_norm},
there exists a constant $L>0$ such that $\Vert x\Vert_{V}\leq L\Vert x\Vert$,
for all $x\in X$. If $\Vert x\Vert\neq0,$ then $2L\Vert x\Vert>L\Vert x\Vert\geq\Vert x\Vert_{V}$.
Hence $x\in2L\Vert x\Vert V$, which is also trivially true if $x=0$.
In particular, for all $x\in X$ there exists some $c\in B$ such
that $x=2L\Vert x\Vert T(c)$. Then $x=T(2L\Vert x\Vert c)$, and
$d(0,2L\Vert x\Vert c)=2L\Vert x\Vert d(0,c)\leq2L\Vert x\Vert$. 

Next, we prove that (2) implies (3). Let $U\subseteq C$ be an open
set, and let $x\in T(U)$ be arbitrary with $b\in U$ satisfying $Tb=x$.
Since $U$ is open, there exists some $r>0$ such that $W:=\{c\in C:d(b,c)<r\}$
is contained in $U$. We define $W':=\{c\in C:d(0,c)<r\}$. Then $b+W'\subseteq W$,
since $d(b,b+w^{\prime})\leq d(0,w^{\prime})<r$ for all $w^{\prime}\in W'$.
Now, by hypothesis, for every $x'\in X$ with $\|x'\|<rK^{-1}$, there
exists some $w^{\prime}\in W'$ with $Tw^{\prime}=x'$. With $B_{X}:=\{x\in X:\|x\|<1\}$,
by additivity of $T$, it follows that $x+rK^{-1}B_{X}\subseteq T(b+W')\subseteq T(W)\subseteq T(U)$.
We conclude that $T(U)$ is open.

That (3) implies (4) is trivial, and (4) implies (1) by the positive
homogeneity of $T$.\end{proof}
\begin{rem}
{}
\begin{enumerate}
\item Since a real or complex Banach space is a complete metric cone, Theorem
\ref{thm:open_mapping_theorem} generalizes the Open Mapping Theorem
for Banach spaces that was used in the proof of the preparatory Proposition
\ref{prop:equivalent_norm}.
\item If $C$ is a closed cone in a Banach space, $X$ is a topological
vector space, and $T:C\to X$ is continuous, additive and positively
homogeneous, then we can conclude that $T$ is an open map, provided
that we know beforehand that the closure of $\{Tc:c\in C,\|c\|\leq1\}$
is a neighborhood of 0 in $X$. This follows from \citep[Theorem 1]{NgOpenMapping}
. Since we do not have such a hypothesis, this result does not imply
ours. The difference is not only that in our case $T$ is assumed
to be surjective, but, more fundamentally, that our image space is
in fact a Banach space, for which an Open Mapping Theorem is already
known to hold that serves as a stepping stone for the more general
result.
\item In \citep{Valero} an Open Mapping Theorem is established for maps
between two abstract cones in a certain class, provided that we know
beforehand that these maps satisfy a so-called almost-openness condition.
Since we do not have such a hypothesis, this result does not imply
ours. Again the difference lies in the image space: in our context
this is not just a cone, but actually a full Banach space with accompanying
Open Mapping Theorem.

\item A result in the vein of our Theorem \ref{thm:open_mapping_theorem} is the following result of Klee's \citep[Theorem 3.2]{Klee}: Suppose $E$ and $F$ are metrizable topological vector spaces, $C$ is a convex cone in $E$, and $T$ is a continuous additive transformation of $C$ onto $F$ such that $T(U\cap C)=-T(U\cap C)  $ for each member $U$ of a neighborhood basis $\mathcal U$ of $0$ in $E$. Suppose $F$ is of the second category and $C$ is complete. Then $T$\ is open. 
Although similar in nature, it does not imply our Theorem \ref{thm:open_mapping_theorem}, which does not require an enveloping space of the cone or, more importantly, a symmetry property for $T$. Neither does our Theorem \ref{thm:open_mapping_theorem} imply Klee's result, not even if $T$ is additionally required to be positively homogeneous, because, as is well known, the compatible metric on a metrizable topological vector spaces can be assumed to be translation invariant, but there is no guarantee for the homogeneity condition (\ref{eq:positive_homogeneity}) of Definition \ref{def:metric-cone} to hold.

Parenthetically let us remark that, interestingly enough, the symmetry condition in Klee's result is actually equivalent to $T$ being continuous and open. This follows from that result and, for the converse implication, the following observation in a general topological context: Let $E$ be a topological space with subspace $C$, and let $T:C\to F$ be a continuous open (not necessarily surjective) map into  a topological space $F$. Let $p\in C$ and suppose $\sigma:  F\to F$ is a homeomorphism of $F$ of finite order such that $\sigma(T(p))=T(p)$. Then there exists a neighborhood basis $\mathcal U$ of $p$ in $E$ such that $\sigma(T(U\cap C))=T(U\cap C)$, for all $U\in\mathcal U$. To see this, let $V$ be a neighborhood of $p$\ in $E$. Since $T$ is open, $T(V\cap C)$ is a neighborhood of $T(p)$. Let $n$\ be the order of $\sigma$ and put $W=\bigcap_{i=0}^{n-1}\sigma^i(T(V\cap C))$. Then $W$ is a neighborhood of $T(p)$, $W\subseteq T(V\cap C)$\  and $\sigma(W)=W$. Since $T$ is continuous, $T^{-1}(W)$ is a neighborhood of $p$\ in $C$. Let $U^\prime$ be a neighborhood of $p$\ in $E$ such that $U^\prime\cap C=T^{-1}(W)$. Put $U=U^\prime\cap V$, so that $U$ is a neighborhood of $p$ in $E$ and $U\subseteq V$. Then $T(U\cap C)=T(U^\prime\cap V\cap C)=T(T^{-1}(W)\cap V\cap C)\subseteq T(T^{-1}(W))\cap T(V\cap C)\subseteq W\cap T(V\cap C)=W$, where the final equality holds since $W\subseteq T(V\cap C)$ by construction. The latter relation also implies that $T(U\cap C)=T(T^{-1}(W)\cap V\cap C)\supseteq W$. Hence $T(U\cap C)=W$ is mapped into and onto itself by $\sigma$, and we conclude that the collection of all $U$ thus obtained forms a neighborhood basis $\mathcal U$ of $p$\ in $E$ as required.

\end{enumerate}
\end{rem}
We will now proceed with the second basic result, Proposition \ref{prop:continuous_right_inverses_Frechet},
which is concerned with families of continuous right inverses for
a surjective (this follows from the hypotheses) map. 
\begin{prop}
\label{prop:continuous_right_inverses_Frechet}Let $X$ be a real
or complex normed space and let $Y$ be a real or complex topological
vector space, not necessarily Hausdorff and not necessarily over the
same field as $X$, with $C\subseteq Y$ a closed not necessarily
proper cone. Let $I$ be a finite set, possibly empty. For each $i\in I$,
let $\alpha_{i}\in\mathbb{R}$ and let $\rho_{i}:C\to\mathbb{R}$
be a continuous subadditive positively homogeneous map. 

Suppose that $T:C\to X$ is a continuous additive positively homogeneous
map with the property that, for every $\varepsilon>0$, there exists
a map $\sigma_{\varepsilon}:S_{X}\to C$, such that:
\begin{enumerate}
\item $T\circ\sigma_{\varepsilon}=\mbox{id}_{S_{X}}$;
\item $\rho_{i}(\sigma_{\varepsilon}(x))\leq\alpha_{i}+\varepsilon$, for
all $x\in S_{X}$ and $i\in I$;
\item $\sigma_{\varepsilon}(S_{X})$ is bounded in $Y$.
\end{enumerate}
Then, for every $\varepsilon>0$, the correspondence $\varphi_{\varepsilon}:S_{X}\twoheadrightarrow Y$,
defined by 
\[
\varphi_{\varepsilon}(x):=\{y\in C:Ty=x,\,\rho_{i}(y)\leq\alpha_{i}+\varepsilon\mbox{ for all }i\in I\}\quad(x\in S_{X}),
\]
has non-empty closed convex values, and is lower hemicontinuous on
$S_{X}$. 

If $Y$ is a Fr\'echet space, there exist continuous maps $\sigma_{\varepsilon}^{\prime}:S_{X}\to C$,
for all $\varepsilon>0$, satisfying:
\begin{enumerate}
\item [\textup{(}a\textup{)}]$T\circ\sigma_{\varepsilon}^{\prime}=\mbox{id}_{S_{X}}$;
\item [\textup{(}b\textup{)}]$\rho_{i}(\sigma_{\varepsilon}^{\prime}(x))\leq\alpha_{i}+\varepsilon$,
for all $x\in S_{X}$ and $i\in I$.
\end{enumerate}
If $\varepsilon>0$ and $\sigma_{\varepsilon}^{\prime}(S_{X})$ is
bounded in \textup{$Y$ }in the sense of topological vector spaces,
then $\sigma_{\varepsilon}^{\prime}$ can be extended to a continuous
positively homogeneous map $\sigma_{\varepsilon}^{\prime}:X\to C$
on the whole space, satisfying: 
\begin{enumerate}
\item [\textup{(}a\textup{)}]$T\circ\sigma_{\varepsilon}^{\prime}=\mbox{id}_{X};$ 
\item [\textup{(}b\textup{)}]$\rho_{i}(\sigma_{\varepsilon}^{\prime}(x))\leq(\alpha_{i}+\varepsilon)\|x\|$
, for all $x\in X$ and $i\in I$. 
\end{enumerate}
\end{prop}
Before embarking on the proof, let us point out that the salient point
lies in the fact that the right inverses $\sigma_{\varepsilon}^{\prime}$
of $T$ on the unit sphere of $X$ are continuous, whereas this is
not required for the original family of the $\sigma_{\varepsilon}$'s,
and that this extra property can be achieved retaining the relevant
inequalities. It is for this that Michael's Selection Theorem is used.
The subsequent conditional extension of such a $\sigma_{\varepsilon}^{\prime}$
to the whole space is rather trivial. 

Furthermore, we note that, although in the applications we have in
mind the constants $\alpha_{i}$ will be positive and each $\rho_{i}$
will be the restriction to $C$ of a continuous seminorm or (if $Y$
is a Banach space) a continuous real-linear functional on the whole
space $Y$, the present proof does not require this.
\begin{proof}
Let $\varepsilon>0$ be arbitrary. Since $\sigma_{\varepsilon}(x)\in\varphi_{\varepsilon}(x)$,
for all $x\in S_{X}$, $\varphi_{\varepsilon}$ is non-empty-valued.
By continuity of $T$ and the $\rho_{i}$'s, $\varphi_{\varepsilon}$
is closed-valued. Since $T$ is affine on the convex set $C$, and
each $\rho_{i}$, if any, is subadditive and positively homogeneous,
$\varphi_{\varepsilon}$ is convex-valued. 

We will now show that $\varphi_{\varepsilon}$ is lower hemicontinuous,
for any fixed $\varepsilon>0$. To this end, let $x\in S_{X}$ be
arbitrary, and let $U\subseteq Y$ be open such that $\varphi_{\varepsilon}(x)\cap U\neq\emptyset$. 

We start by establishing that there exists some $y\in\varphi_{\varepsilon}(x)\cap U$
such that $\rho_{i}(y)<\alpha_{i}+\varepsilon$, for all $i\in I$
(if any), where the inequality that is valid for $\sigma_{\varepsilon}($$x)$
has been improved to strict inequality for $y$. As to this, choose
$y^{\prime}$ in the non-empty set $\varphi_{\varepsilon}(x)\cap U$,
and define $y_{t}:=t\sigma_{\varepsilon/2}(x)+(1-t)y^{\prime}$, for
$t\in[0,1]$. Then $y_{t}\in C$ and $Ty_{t}=x$, for all $t\in[0,1]$.
Now, for all $t\in(0,1]$ and all $i\in I$, 
\begin{eqnarray*}
\rho_{i}(y_{t}) & \leq & t\rho_{i}(\sigma_{\varepsilon/2}(x))+(1-t)\rho_{i}(y^{\prime})\\
 & \leq & t\left(\alpha_{i}+\frac{\varepsilon}{2}\right)+(1-t)\left(\alpha_{i}+\varepsilon\right)\\
 & < & t\left(\alpha_{i}+\varepsilon\right)+(1-t)\left(\alpha_{i}+\varepsilon\right)\\
 & = & \left(\alpha_{i}+\varepsilon\right).
\end{eqnarray*}
Since $U$ is open, there exists some $t_{0}>0$ such that $y:=y_{t_{0}}\in U$.
Then $y$ is as required.

Having found and fixed this $y$, we define $\eta:=\min_{i\in I}\{\alpha_{i}+\varepsilon-\rho_{i}(y)\}>0$
if $I\neq\emptyset$, and $\eta:=1$ if $I=\emptyset$; here we use
that $I$ is finite.

Next, let $W\subseteq Y$ be an open neighborhood of zero such that
$y+W\subseteq U$. Since $\sigma_{\eta/2}(S_{X})$ is bounded by assumption,
we can fix some $0<r\leq1$ such that $r'\sigma_{\eta/2}(S_{X})\subseteq W$,
for all $0\leq r^{\prime}<r$, and $\alpha_{i}r'<\eta/2$, for all
$0\leq r^{\prime}<r$ and all $i\in I$. 

Now, if $x'\in X$ satisfies $0<\|x'\|<r$, then
\[
y+\|x'\|\sigma_{\eta/2}\left(\frac{x'}{\|x'\|}\right)\in C,
\]
\[
T\left(y+\|x'\|\sigma_{\eta/2}\left(\frac{x'}{\|x'\|}\right)\right)=x+x',
\]
and 
\[
y+\|x'\|\sigma_{\eta/2}\left(\frac{x'}{\|x'\|}\right)\in y+\Vert x^{\prime}\Vert\sigma_{\eta/2}(S_{X})\subset y+W\subset U.
\]
Furthermore, for such $x^{\prime}$ and for $i\in I$ we find, using
$\alpha_{i}\|x'\|<\eta/2$ and $\|x'\|<r\leq1$, that 
\begin{eqnarray*}
\rho_{i}\left(y+\|x'\|\sigma_{\eta/2}\left(\frac{x'}{\|x'\|}\right)\right) & \leq & \rho_{i}(y)+\|x'\|\rho_{i}\left(\sigma_{\eta/2}\left(\frac{x'}{\|x'\|}\right)\right)\\*
 & \leq & \rho_{i}(y)+\|x'\|\left(\alpha_{i}+\frac{\eta}{2}\right)\\
 & < & \rho_{i}(y)+\frac{\eta}{2}+\frac{\eta}{2}\\
 & = & \rho_{i}(y)+\eta\\
 & \leq & \rho_{i}(y)+\alpha_{i}+\varepsilon-\rho_{i}(y)\\
 & = & \alpha_{i}+\varepsilon.
\end{eqnarray*}
Therefore, if $x+x'\in S_{X}$ with $0<\|x'\|<r$, we conclude that
\[
y+\|x'\|\sigma_{\eta/2}(x'/\|x'\|)\in\varphi_{\varepsilon}(x+x')\cap U.
\]
Hence $\varphi_{\varepsilon}$ is lower hemicontinuous on $S_{X}$,
as was to be proved.

If $Y$ is a Fr\'echet space then, since $S_{X}$ as a metric space
is paracompact \citep{Stone}, Michael's Selection Theorem (Theorem
\ref{thm:selection_theorem}), applied to each individual $\varphi_{\varepsilon}$,
supplies a family of continuous maps $\sigma_{\varepsilon}^{\prime}:S_{X}\to C$,
such that $\sigma_{\varepsilon}(x)\in\varphi_{\varepsilon}(x)$, for
all $\varepsilon>0$ and all $x\in S_{X}$. Then the $\sigma_{\varepsilon}^{\prime}$
are as required.

If $\varepsilon>0$ and $\sigma_{\varepsilon}^{\prime}(S_{X})$ happens
to be bounded in the topological vector space $Y$, we extend $\sigma_{\varepsilon}^{\prime}:S_{X}\to C$
to a positively homogeneous $C$-valued map on all of $X$, also denoted
by $\sigma_{\varepsilon}^{\prime}$, by defining 
\[
\sigma_{\varepsilon}^{\prime}(x):=\begin{cases} 
0 & \mbox{\mbox{{for\,\,}}}x=0;\\
\|x\|\sigma_{\varepsilon}^{\prime}\left(\frac{x}{\|x\|}\right) & \mbox{{for\,\,}}x\neq0.
\end{cases}
\]
The continuity of $\sigma_{\varepsilon}^{\prime}$ at 0 then follows
from the boundedness of $\sigma_{\varepsilon}^{\prime}(S_{X})$, and
at other points it is immediate. It is easily verified that such a
global $\sigma_{\varepsilon}^{\prime}$ has the properties as claimed.
\end{proof}
Combination of Theorem \ref{thm:open_mapping_theorem} and Proposition
\ref{prop:continuous_right_inverses_Frechet} yields the following
key result on right inverses of surjections from cones onto Banach
spaces. The structure of the proofs makes it clear that it is ultimately
based on the Open Mapping Theorem for Banach spaces and Michael's
Selection Theorem.
\begin{prop}
\label{prop:key_proposition}Let $X$ and $Y$ be real or complex
Banach spaces, not necessarily over the same field, and let $C$ be
a closed not necessarily proper cone in $Y$. Let $T:C\to X$ be a
surjective continuous additive positively homogeneous map. 

Furthermore, let $J$ be a finite set, possibly empty, and, for all
$j\in J$, let $\rho_{j}:C\to\mathbb{R}$ be a continuous subadditive
positively homogeneous map. For example, each $\rho_{j}$ could be
the restriction to $C$ of a globally defined continuous seminorm
or continuous real-linear functional.
\begin{enumerate}
\item If $\rho_{j}$ is bounded from above on $S_{Y}\cap C$, for all $j\in J$,
then there exist constants $K>0$, $\alpha_{j}\in\mathbb{\mathbb{R}}$
$(j\in J)$, and a map $\gamma:S_{X}\to C$, such that:

\begin{enumerate}
\item $T\circ\gamma=\mbox{id}_{S_{X}}$;
\item $\Vert\gamma(x)\Vert\leq K$, for all $x\in S_{X}$;
\item $\rho_{j}(\gamma(x))\leq\alpha_{j}$, for all $x\in S_{X}$ and $j\in J$.
\end{enumerate}
\item If $K>0,\,\alpha_{j}\in\mathbb{R}$ $(j\in J)$, and $\gamma:S_{X}\to C$
satisfy \textup{(a)}, \textup{(b)} and \textup{(c)} in part \textup{(1)},
then, for every $\varepsilon>0$, there exists a continuous positively
homogeneous map $\gamma_{\varepsilon}:X\to C$ such that:

\begin{enumerate}
\item $T\circ\gamma_{\varepsilon}=\mbox{id}_{X}$;
\item $\Vert\gamma_{\varepsilon}(x)\Vert\leq(K+\varepsilon)\Vert x\Vert$,
for all $x\in X$;
\item $\rho_{j}(\gamma(x))\leq(\alpha_{j}+\varepsilon)\Vert x\Vert$, for
all $x\in X$ and $j\in J$.
\end{enumerate}
\end{enumerate}
\end{prop}
\begin{proof}
As to the first part, we start by applying part (2) of Theorem \ref{thm:open_mapping_theorem}
and obtain $K>0$ and a map $\gamma:S_{X}\to C$, such that $T\circ\gamma=\mbox{id}_{S_{X}}$
and $\Vert\gamma(x)\Vert\leq K$ $(x\in S_{X})$. 

If $j\in J$, and $\beta_{j}\in\mathbb{R}$ is such that $\rho_{j}(c)\leq\beta_{j}$
for all $c\in S_{Y}\cap C$, where we may assume that $\beta_{j}\geq0$,
then $\rho_{j}(\gamma(x))\leq K\beta_{j}$, for all $x\in S_{X}$.
Indeed, this is obvious if $\gamma(x)$=0, and if $\gamma(x)\neq0$
we have
\[
\rho_{j}\left(\gamma(x)\right)=\Vert\gamma(x)\Vert\rho_{j}\left(\frac{\gamma(x)}{\|\gamma(x)\|}\right)\leq K\beta_{j}.
\]

The existence of the $\alpha_{j}:=K\beta_{j}$ is then clear. 

As to the second part, suppose that $K>0,\,\alpha_{j}\in\mathbb{R}$
$(j\in J$) and $\gamma:S_{X}\to C$ satisfy (a), (b) and (c) in part
(1). We augment J to $I:=J\cup\{\Vert\,\Vert\}$, where we choose
an index symbol $\Vert\,\Vert\notin J$, and let $\rho_{\Vert\,\Vert}(c):=\Vert c\Vert$,
for $c\in C$, and put $\alpha_{\Vert\,\Vert}:=K$. We can now apply
Proposition \ref{prop:continuous_right_inverses_Frechet} with $\sigma_{\varepsilon}=\gamma$
for all $\varepsilon>0$, since its hypotheses (1), (2) and (3) are
then satisfied. The continuous $\sigma_{\varepsilon}^{\prime}:S_{X}\to C$
as furnished by Proposition \ref{prop:continuous_right_inverses_Frechet}
are, in particular, such that $\rho_{\Vert\,\Vert}(\sigma_{\varepsilon}^{\prime}(x))\leq\alpha_{\Vert\,\Vert}+\varepsilon$,
i.e., such that $\Vert\sigma_{\varepsilon}^{\prime}(x)\Vert\leq K+\varepsilon$,
for all $x\in S_{X}$. Hence each of the sets $\sigma_{\varepsilon}^{\prime}(S_{X})$
is bounded in $Y$, and the last part of Proposition \ref{prop:continuous_right_inverses_Frechet}
applies, yielding global $\sigma_{\varepsilon}^{\prime}:X\to C$ that
can be taken as the required $\gamma_{\varepsilon}.$
\end{proof}
Let us remark explicitly that the $\alpha_{j}$'s need not be non-negative
and that in part (2) the $\rho_{j}$'s are not required to be bounded
from above on $S_{Y}\cap C$ as in part (1), but rather on $\gamma(S_{X})$
(as a consequence of the hypothesized validity of (1)(c)), which is
a weaker hypothesis. 

We extract two practical consequences from Proposition \ref{prop:key_proposition}.
First of all, if the $\rho_{j}$'s are bounded from above on $S_{Y}\cap C$
then part (1) of Proposition \ref{prop:key_proposition} is applicable
and its conclusion shows that the hypothesis of part (2) are satisfied.
Taking $\varepsilon=1$, say, we therefore have the following, which has the Strong Open Mapping Theorem (Theorem \ref{thm:bounded_continuous_inverses_exist_introduction_version}) as a special case when $J$ is empty.
\begin{thm}
\label{thm:continuous_right_inverses_exists}
\textup{(}Strong Open Mapping Theorem, General Version\textup{)}
Let $X$ and $Y$ be
real or complex Banach spaces, not necessarily over the same field,
and let $C$ be a closed not necessarily proper cone in $Y$. Let
$T:C\to X$ be a surjective continuous additive positively homogeneous
map.

Furthermore, let $J$ be a finite set, possibly empty, and, for all
$j\in J$, let $\rho_{j}:C\to\mathbb{R}$ be a continuous subadditive
positively homogeneous map that is bounded from above on $S_{Y}\cap C$.
For example, each $\rho_{j}$ could be the restriction to $C$ of
a globally defined continuous seminorm or continuous real-linear functional.

Then there exist constants $K>0$ and $\alpha_{j}\in\mathbb{R}\,\,(j\in J)$
and a continuous positively homogeneous map $\gamma:X\to C$, such
that:
\begin{enumerate}
\item $T\circ\gamma=\mbox{id}_{X}$;
\item $\Vert\gamma(x)\Vert\leq K\Vert x\Vert$, for all $x\in X$;
\item $\rho_{j}(\gamma(x))\leq\alpha_{j}\Vert x\Vert$, for all $x\in X$
and $j\in J$. 
\end{enumerate}
\end{thm}
The next consequence of Proposition \ref{prop:key_proposition} states
that the existence of a family of possibly ill-behaved right inverses
on the unit sphere is actually equivalent with the existence of a
family of well-behaved global ones. Note that, compared with Theorem
\ref{thm:continuous_right_inverses_exists}, the boundedness assumption
from above for the $\rho_{j}$'s on $S_{Y}\cap C$ has been replaced
with the assumptions (1)(c) and (2)(c) below.
\begin{thm}
\label{thm:improving_a_family_of_right_inverses_is_possible}Let $X$
and $Y$ be real or complex Banach spaces, not necessarily over the
same field, and let $C$ be a closed not necessarily proper cone in
$Y$. Let $T:C\to X$ be a surjective continuous additive positively
homogeneous map. 

Furthermore, let $J$ be a finite set, possibly empty, and, for all
$j\in J$, let $\rho_{j}:C\to\mathbb{R}$ be a continuous subadditive
positively homogeneous map. For example, each $\rho_{j}$ could be
the restriction to $C$ of a globally defined continuous seminorm
or continuous real-linear functional.

If $K>0$ and $\alpha_{j}\in\mathbb{R}\,(j\in J)$ are constants,
then the following are equivalent:
\begin{enumerate}
\item For every $\varepsilon>0$, there exists a map $\gamma_{\varepsilon}:S_{X}\to C$,
such that:

\begin{enumerate}
\item $T\circ\gamma_{\varepsilon}=\mbox{id}_{S_{X}}$;
\item $\Vert\gamma_{\varepsilon}(x)\Vert\leq K+\varepsilon$, for all $x\in S_{X}$;
\item $\rho_{j}(\gamma_{\varepsilon}(x))\leq\alpha_{j}+\varepsilon$, for
all $x\in S_{X}$ and $j\in J$.
\end{enumerate}
\item For every $\varepsilon>0$, there exists a continuous positively homogeneous
map $\gamma_{\varepsilon}:X\to C$ such that:

\begin{enumerate}
\item $T\circ\gamma_{\varepsilon}=\mbox{id}_{X}$;
\item $\Vert\gamma_{\varepsilon}(x)\Vert\leq(K+\varepsilon)\Vert x\Vert$,
for all $x\in X$;
\item $\rho_{j}(\gamma_{\varepsilon}(x))\leq(\alpha_{j}+\varepsilon)\Vert x\Vert$,
for all $x\in X$ and $j\in J$.
\end{enumerate}
\end{enumerate}
\end{thm}
\begin{proof}
Clearly the second part implies the first. For the converse implication,
let $\varepsilon>0$ be given. Then, by assumption, there exists a
map (we add accents to avoid notational confusion) $\gamma_{\varepsilon/2}^{\prime}:S_{X}\to C$,
such that:
\begin{enumerate}
\item $T\circ\gamma_{\varepsilon/2}^{\prime}=\mbox{id}_{S_{X}}$;
\item $\Vert\gamma_{\varepsilon/2}^{\prime}(x)\Vert\leq K+\varepsilon/2$,
for all $x\in S_{X}$;
\item $\rho_{j}(\gamma_{\varepsilon/2}^{\prime}(x))\leq\alpha_{j}+\varepsilon/2$,
for all $x\in S_{X}$ and $j\in J$.
\end{enumerate}
We can now apply part (2) of Proposition \ref{prop:key_proposition},
with $K$ replaced with $K+\varepsilon/2$, $\alpha_{j}$ with $\alpha_{j}+\varepsilon/2$,
$\gamma$ with $\gamma_{\varepsilon/2}^{\prime}$, and $\varepsilon$
with $\varepsilon/2$. The map $\gamma_{\varepsilon/2}$ as furnished
by part (2) of Proposition \ref{prop:key_proposition} can then be
taken as the map $\gamma_{\varepsilon}$ in part (2) of the present
Theorem.
\end{proof}

\section{Applications\label{sec:applications}}

By varying $C$ and the $\rho_{j}$'s various types of consequences
of Theorems \ref{thm:continuous_right_inverses_exists} and \ref{thm:improving_a_family_of_right_inverses_is_possible}
can be obtained, and we collect some in the present section, considering
situations where the $\rho_{j}$'s are restrictions to $C$ of globally
defined continuous seminorms or continuous real-linear functionals.
This seems to be a natural context to work in, but we note that it
is not required as such by these two underlying Theorems, nor by the
key Proposition \ref{prop:key_proposition}, so that applications
of another type are conceivable. 

As in earlier sections, if $X$ is a normed space, then $S_{X}:=\{x\in X:\Vert x\Vert=1\}$
is its unit sphere.

To start with, Theorems \ref{thm:continuous_right_inverses_exists}
and \ref{thm:improving_a_family_of_right_inverses_is_possible} are
clearly applicable when $T:C\to X$ is the restriction to $C$ of
a global continuous linear map $T:Y\to X$ and (as already mentioned
in these Theorems) each of the $\rho_{j}$'s is the restriction of
a globally defined continuous seminorm or continuous real-linear functional.
Furthermore, $Y$ is a closed cone in itself, so that these Theorems
can be specialized to yield statements on well-behaved right inverses
for continuous linear surjections between Banach spaces. For reasons
of space, we refrain from explicitly formulating all these quite obvious
special cases.

Instead, we give applications to the internal structure of a Banach
space that is a sum of closed not necessarily proper cones, and to
the structure of spaces of continuous functions with values in such
a Banach space. Thus we return to the to the improvements of And\^o's
Theorem and our original motivating question alluded to in the introduction.

The following result applies, in particular, when $X=\sum_{i=1}^{n}C_{i}$
is the sum of a finite number of closed not necessarily proper cones.
In that case, the Banach space $Y$ in the following Theorem is the
direct sum of $n$ copies of $X$.
\begin{thm}
\label{thm:continuous_selections_from_generating_cones_ell_one_case}Let
$X$ be a real or complex Banach space. Let $I$ be a non-empty set,
possibly uncountable, and let $\{C_{i}\}_{i\in I}$ be a collection
of closed not necessarily proper cones in $X$, such that every $x\in X$
can be written as an absolutely convergent series $x=\sum_{i\in I}c_{i}$,
where $c_{i}\in C_{i}$, for all $i\in I.$ 

Let $Y=\ell^{1}(I,X)$ be the $\ell^{1}$-direct sum of $|I|$ copies
of $X$, and let $C$ be the natural closed cone in the Banach space
$Y$, consisting of those elements where the $i$-th component is
in $C_{i}$. Finally, let $J$ be a finite set, possibly empty, and,
for all $j\in J$, let $\rho_{j}:Y\to\mathbb{R}$ be a continuous
seminorm or a continuous real-linear functional. 

Then:
\begin{enumerate}
\item There exist a constant $K>0$ and a continuous positively homogeneous
map $\gamma:X\to C$ with continuous positively homogeneous component
maps $\gamma_{i}:X\to C_{i}\,(i\in I)$, such that:

\begin{enumerate}
\item $x=\sum_{i\in I}\gamma_{i}(x)$, for all $x\in X$;
\item $\sum_{i\in I}\Vert\gamma_{i}(x)\Vert\leq K\Vert x\Vert$, for all
$x\in X$.
\end{enumerate}
\item If $K>0$ and $\alpha_{j}\in\mathbb{R\,}(j\in J)$ are constants,
then the following are equivalent:

\begin{enumerate}
\item For every $\varepsilon>0$, there exists a map $\gamma_{\varepsilon}:S_{X}\to C$
with component maps $\gamma_{\varepsilon,i}:S_{X}\to C_{i}\,(i\in I)$,
such that:

\begin{enumerate}
\item $x=\sum_{i\in I}\gamma_{\varepsilon,i}(x)$, for all $x\in S_{X}$;
\item $\sum_{i\in I}\Vert\gamma_{\varepsilon,i}(x)\Vert\leq(K+\varepsilon)$,
for all $x\in S_{X}$;
\item $\rho_{j}(\gamma_{\varepsilon}(x))\leq(\alpha_{j}+\varepsilon)$,
for all $x\in S_{X}$ and $j\in J$.
\end{enumerate}
\item For every $\varepsilon>0$, there exists a continuous positively homogeneous
map $\gamma_{\varepsilon}:X\to C$ with continuous positively homogeneous
component maps $\gamma_{\varepsilon,i}:X\to C_{i}\,(i\in I)$, such
that:

\begin{enumerate}
\item $x=\sum_{i\in I}\gamma_{\varepsilon,i}(x)$, for all $x\in X$;
\item $\sum_{i\in I}\Vert\gamma_{\varepsilon,i}(x)\Vert\leq(K+\varepsilon)\Vert x\Vert$,
for all $x\in X$;
\item $\rho_{j}(\gamma_{\varepsilon}(x))\leq(\alpha_{j}+\varepsilon)\Vert x\Vert$,
for all $x\in X$ and $j\in J$.
\end{enumerate}
\end{enumerate}
\end{enumerate}
\end{thm}
\begin{proof}
Let $T:C\to X$ be the canonical summing map. Then Theorem \ref{thm:continuous_right_inverses_exists}
yields part (1), and Theorem \ref{thm:improving_a_family_of_right_inverses_is_possible}
yields part (2).
\end{proof}
In order to illustrate Theorem \ref{thm:continuous_selections_from_generating_cones_ell_one_case}
we consider the situation where $X$ is a Banach space, (pre)-ordered
by a closed not necessarily proper cone $X^{+}$. If $X^{+}$ is generating
in the sense of (pre)-ordered Banach spaces, i.e., if $X=X^{+}-X^{+}$,
and if $X^{+}$ is proper, then And\^o's Theorem (Theorem \ref{thm:Ando_theorem})
applies. On the other hand, Theorem \ref{thm:continuous_selections_from_generating_cones_ell_one_case},
also yields this result (and an even stronger one) by writing $X=X^{+}+(-X^{+})$
as the sum of two closed cones, coincidentally related by a minus
sign. For convenience we formulate the result explicitly in the usual
notation with minus signs.
\begin{cor}
\label{cor:ando_improved}Let $X$ be a real \textup{(}pre\textup{)}-ordered
Banach space, \textup{(}pre\textup{)}-ordered by a closed generating
not necessarily proper cone $X^{+}$. Let $J$ be a finite set, possibly
empty, and, for all $j\in J$, let $\rho_{j}:X\times X\to\mathbb{R}$
be a continuous seminorm or a continuous linear functional. Then:
\begin{enumerate}
\item There exist a constant $K>0$ and continuous positively homogeneous
maps $\gamma^{\pm}:X\to X^{+}$, such that:

\begin{enumerate}
\item $x=\gamma^{+}(x)-\gamma^{-}(x)$, for all $x\in X$;
\item $\Vert\gamma^{+}(x)\Vert+\Vert\gamma^{-}(x)\Vert\leq K\Vert x\Vert$,
for all $x\in X$.
\end{enumerate}
\item If $K>0$ and $\alpha_{j}\in\mathbb{R\ }(j\in J)$ are constants,
then the following are equivalent:

\begin{enumerate}
\item For every $\varepsilon>0$, there exist maps $\gamma_{\varepsilon}^{\pm}:S_{X}\to X^{+}$,
such that:

\begin{enumerate}
\item $x=\gamma_{\varepsilon}^{+}(x)-\gamma_{\varepsilon}^{-}(x)$, for
all $x\in S_{X}$;
\item $\Vert\gamma_{\varepsilon}^{+}(x)\Vert+\Vert\gamma_{\varepsilon}^{-}(x)\Vert\leq(K+\varepsilon)$,
for all $x\in S_{X}$;
\item $\rho_{j}((\gamma_{\varepsilon}^{+}(x),\gamma_{\varepsilon}^{-}(x))\leq(\alpha_{j}+\varepsilon)$,
for all $x\in S_{X}$ and $j\in J$.
\end{enumerate}
\item For every $\varepsilon>0$, there exist continuous positively homogeneous
maps $\gamma_{\varepsilon}^{\pm}:X\to X^{+}$, such that:

\begin{enumerate}
\item $x=\gamma_{\varepsilon}^{+}(x)-\gamma_{\varepsilon}^{-}(x)$, for
all $x\in X$;
\item $\Vert\gamma_{\varepsilon}^{+}(x)\Vert+\Vert\gamma_{\varepsilon}^{-}(x)\Vert\leq(K+\varepsilon)\|x\|$,
for all $x\in X$;
\item $\rho_{j}((\gamma_{\varepsilon}^{+}(x),\gamma_{\varepsilon}^{-}(x))\leq(\alpha_{j}+\varepsilon)\|x\|$,
for all $X$ and $j\in J$.
\end{enumerate}
\end{enumerate}
\end{enumerate}
\end{cor}
\begin{proof}
We apply Theorem \ref{thm:continuous_selections_from_generating_cones_ell_one_case}
with $I=\{1,2\}$, $C_{1}=X^{+}$ and $C_{2}:=-X^{+}$, and then let
$\gamma^{+}=\gamma_{1}$ and $\gamma^{-}=-\gamma_{2}$ in part (1),
and $\gamma_{\varepsilon}^{+}=\gamma_{\varepsilon,1}$ and $\gamma_{\varepsilon}^{-}=-\gamma_{\varepsilon,2}$
in part (2)
\end{proof}
We continue in the context of a real (pre)-ordered normed space $X$
ordered by a closed generating not necessarily proper cone $X^{+}$.
If $\alpha>0$, then we will say that $X$ is
\begin{enumerate}
\item $\alpha$-\textit{conormal} if, for each $x\in X$, there exist $x^{\pm}\in X^{+}$,
such that $x=x^{+}-x^{-}$ and $\Vert x^{+}\Vert\leq\alpha\Vert x\Vert$;
\item \textit{approximately $\alpha$-conormal} if $X$ is $(\alpha+\varepsilon)$-conormal,
for all $\varepsilon>0.$
\end{enumerate}
And\^o's Theorem is equivalent to asserting that every real Banach
space, ordered by a closed generating proper cone, is $\alpha$-conormal
for some $\alpha>0$. Clearly, $\alpha$-conormality implies approximate
$\alpha$-conormality. What is less obvious is that approximate $\alpha$-conormality
is equivalent with a continuous positively homogeneous version of
the same notion, as is the content of the following consequence of
Corollary \ref{cor:ando_improved}. 
\begin{cor}
\label{cor:continous_approximate_conormality} Let $X$ be a real
\textup{(}pre\textup{)}-ordered Banach space, \textup{(}pre\textup{)}-ordered
by a closed generating not necessarily proper cone $X^{+}$, and let
$\alpha>0.$ Then the following are equivalent: 
\begin{enumerate}
\item $X$ is approximately $\alpha$-conormal;
\item For every $\varepsilon>0$, there exist continuous positively homogeneous
maps $\gamma_{\varepsilon}^{\pm}:X\to X^{+}$, such that:

\begin{enumerate}
\item $x=\gamma_{\varepsilon}^{+}(x)-\gamma_{\varepsilon}^{-}(x)$, for
all $x\in X$;
\item $\Vert\gamma_{\varepsilon}^{+}(x)\Vert\leq(\alpha+\varepsilon)\Vert x\Vert$,
for all $x\in X$.
\end{enumerate}
\end{enumerate}
\end{cor}
\begin{proof}
Clearly part (2) implies part (1). For the converse we will apply
Corollary \ref{cor:ando_improved} with $J=\{1\}$, as follows. Let
$\varepsilon>0$ be given and fixed. For each $x\in S_{X}$, the $(\alpha+\varepsilon/2)$-conormality
of $X$ implies that, for each $x\in S_{X}$, we can choose and fix
$\gamma_{\varepsilon/2}^{\pm}(x)\in X^{+}$, such that
\[
x=\gamma_{\varepsilon/2}^{+}(x)-\gamma_{\varepsilon/2}^{-}(x)\,\,(x\in S_{X}),
\]
and $\Vert\gamma_{\varepsilon/2}^{+}(x)\Vert\leq(\alpha+\varepsilon/2)$.
Then $\Vert\gamma_{\varepsilon/2}^{-}(x)\Vert\leq(\alpha+\varepsilon/2+1)$,
so that
\[
\Vert\gamma_{\varepsilon/2}^{+}(x)\Vert+\Vert\gamma_{\varepsilon/2}^{-}(x)\Vert\leq(2\alpha+1+\varepsilon)\,\,(x\in S_{X}).
\]
Define $\rho_{1}:X\times X\to\mathbb{R}$ by $\rho_{1}((x_{1},x_{2})):=\Vert x_{1}\Vert$,
for $x_{1},x_{2}\in X$, so that 
\[
\rho_{1}((\gamma_{\varepsilon/2}^{+}(x),\gamma_{\varepsilon/2}^{-}(x))=\Vert\gamma_{\varepsilon/2}^{+}(x)\Vert\leq\alpha+\varepsilon/2\leq\alpha+\varepsilon\,\,(x\in S_{X}).
\]

Thus we have found constants $K=2\alpha+1>0$, $\alpha_{1}=\alpha$
and maps $\gamma_{\varepsilon/2}^{\pm}:S_{X}\to X_{+}$ satisfying
(i), (ii) and (iii) in part (2)(a) of Corollary \ref{cor:ando_improved}.
Hence the continuous positively homogeneous maps as in part (2)(b)
of Corollary \ref{cor:ando_improved} also exist, and these are as
required. \end{proof}
\begin{rem}
The term ``conormality'' is due to Walsh \citep{Walsh} and several
variations of it have been studied. For example, a real normed space
$X$ is said to be \textit{$\alpha$-max-conormal} if, for each $x\in X$,
there exist $x^{\pm}\in X^{+}$, such that $x=x^{+}-x^{-}$ and $\max(\Vert x^{+}\Vert,\Vert x^{-}\Vert)\leq\alpha\Vert x\Vert$;
$X$ is \textit{approximately $\alpha$-max-conormal }if it is $(\alpha+\varepsilon$)-max-conormal,
for every $\varepsilon>0$. As another example, $X$ is said to be
\textit{$\alpha$-sum-conormal} if, for each $x\in X$, there exist
$x^{\pm}\in X^{+}$, such that $x=x^{+}-x^{-}$ and $\Vert x^{+}\Vert+\Vert x^{-}\Vert\leq\alpha\Vert x\Vert$;
$X$ is \textit{approximately $\alpha$-sum-conormal} if it is $(\alpha+\varepsilon$)-sum-conormal,
for every $\varepsilon>0$. Just as Corollary \ref{cor:continous_approximate_conormality}
shows that approximately $\alpha$-conormality implies its continuous
positively homogeneous version, the elements $x^{\pm}$ figuring in
the definitions of approximately $\alpha$-max-conormality and approximately
$\alpha$-sum-conormality can be chosen in a continuous and positively
homogeneous fashion. The proof is analogous to the proof of Corollary
\ref{cor:continous_approximate_conormality}, but now taking $\rho_{1}(x_{1},x_{2}):=\max(\|x_{1}\|,\|x_{2}\|)$
for approximately $\alpha$-sum-conormality, and $\rho_{1}(x_{1},x_{2}):=\Vert x_{1}\Vert+\Vert x_{2}\Vert$
for approximately $\alpha$-sum-conormality. 

The dual notion of conormality is normality (terminology due to Krein
\citep{KreinNormality}). Several equivalences between versions of
normality of an ordered Banach space $X$ and conormality of its dual
(and vice versa) are known, but are scattered throughout the literature
under various names. The most complete account of normality-conormality
duality relationships may be found in \citep{BattyRobinson}.
\end{rem}
Finally, we return to our original motivating context in the introduction,
but in a more general setting. As a rule, no additional hypotheses
on the topological space $\Omega$ are necessary to pass from $X$
to a space of $X$-valued functions, since the arguments are pointwise
in $X$, but for some converse implications it is required that the
vector valued function space in question is non-zero. If $C_{c}(\Omega)\neq\{0\}$,
for example if $\Omega$ is a non-empty locally compact Hausdorff
space, then this assumption is always satisfied.
\begin{thm}
\label{thm:continuous_functions_arbitrary_topological_space} Let
$X$ be a real or complex Banach space. Let $I$ be a non-empty set,
possibly uncountable, and let $\{C_{i}\}_{i\in I}$ be a collection
of closed not necessarily proper cones in $X$, such that every $x\in X$
can be written as an absolutely convergent series $x=\sum_{i\in I}c_{i}$,
where $c_{i}\in C_{i}$, for all $i\in I.$ Let $\Omega$ be a topological
space. Then there exists a constant $K>0$ with the property that,
for each $X$-valued continuous function $f\in C(\Omega,X)$ on $\Omega$,
there exist $f_{i}\in C(\Omega,C_{i})\,\,(i\in I)$, such that
\begin{enumerate}
\item For every $\omega\in\Omega$, $f(\omega)=\sum_{i\in I}f_{i}(\omega)$,
and $\sum_{i\in I}\Vert f_{i}(\omega)\Vert\leq K\Vert f(\omega)\Vert;$
\item $\Vert f_{i}\Vert_{\infty}\leq K\Vert f\Vert_{\infty}$, for all $i\in I$,
where the right hand side, or both the left hand side and the right
hand side, may be infinite;
\item The support of each $f_{i}$ is contained in that of $f$;
\item If $f$ vanishes at infinity, then so does each $f_{i}$;
\item If $\omega_{1},\omega_{2}\in\Omega$ and $\lambda_{1},\lambda_{2}\geq0$
are such that $\lambda_{1}f(\omega_{1})=\lambda_{2}f(\omega_{2})$,
then $\lambda_{1}f_{i}(\omega_{1})=\lambda_{2}f_{i}(\omega_{2})$,
for all $i\in I$.
\end{enumerate}
In particular, if $I$ \label{prop:decomposition_for_X_from that_for_functions}is
finite, so that $X=\sum_{i\in I}C_{i},$ then we can write the following
vector spaces as the sum of cones naturally associated with the $C_{i}$,
where the cones are closed in the last three normed spaces:
\begin{enumerate}
\item $C(\Omega,X)=\sum_{i\in I}C(\Omega,C_{i})$ for the continuous $X$-valued
functions on $\Omega$;
\item $C_{b}(\Omega,X)=\sum_{i\in I}C_{b}(\Omega,C_{i})$ for the bounded
continuous $X$-valued functions on $\Omega$;
\item $C_{0}(\Omega,X)=\sum_{i\in I}C_{0}(\Omega,C_{i})$ for the continuous
$X$-valued functions on $\Omega$ vanishing at infinity;
\item $C_{c}(\Omega,X)=\sum_{i\in I}C_{c}(\Omega,C_{i})$ for the compactly
supported continuous $X$-valued functions on $\Omega$.
\end{enumerate}
\end{thm}
\begin{proof}
We apply part (1) of Theorem \ref{thm:continuous_selections_from_generating_cones_ell_one_case}
and let $f_{i}:=\gamma_{i}\circ f\,\,(i\in I)$. This supplies the
$f_{i}$ as required for the first part, and the statement on the
finite number of naturally associated cones is then clear. 
\end{proof}
Clearly then, the answer to our original question in the introduction
is affirmative: If $X$ is a Banach space with a closed generating
proper cone $X^{+}$, and $\Omega$ is a topological space, then $C_{0}(\Omega,X^{+})$
is generating in $C_{0}(\Omega,X)$. In fact, Theorem \ref{thm:continuous_functions_arbitrary_topological_space}
shows that $X^{+}$ need not even be proper. 
\begin{rem}
\label{rem:asmiow-atkinson-wikstead}As mentioned in the introduction,
if $\Omega$ is a (locally) compact Hausdorff space and $X$ is a
(pre-)ordered Banach space with closed generating cone $X^{+}$, certain
special cases of Theorem \ref{thm:continuous_functions_arbitrary_topological_space}
also follow from \citep[Theorem 2.3]{AsimowAtkinson} and \citep[Theorem 4.4]{Wickstead}.
Both of these results proceed through an application of Lazar's affine
selection theorem to show that cones of continuous affine $X^{+}$-valued
functions on a Choquet simplex $K$ are generating in spaces of continuous
affine $X^{+}$-valued functions on $K$. If $\Omega$ is a compact
Hausdorff space, the fact that $C(\Omega,X^{+})$ is generating in
$C(\Omega,X)$ follows from \citep[Theorem 2.3]{AsimowAtkinson} by
taking $K$ to be the Choquet simplex of all regular Borel probability
measures on $\Omega$, and considering the maps $\mu\mapsto\int_{\Omega}fd\mu\ (\mu\in K,\ f\in C(\Omega,X))$
and $\omega\mapsto a(\delta_{\omega})\ (\omega\in\Omega,\ a\in A(K,X))$,
where $A(K,X)$ denotes the space of continuous affine $X$-valued
functions on $K$.
\end{rem}
The converse of the four last statements in Theorem \ref{thm:continuous_functions_arbitrary_topological_space}
also holds provided the function spaces are non-zero, as is shown
by our next (elementary) result. Note that $C(\Omega,X)$ and $C_{b}(\Omega,X)$
are zero only when $\Omega\neq\emptyset$ and $X=\{0\}$. 
\begin{lem}
\label{lem:decomposition_for_X_from_that_for_functions}Let $X$ be
a real or complex normed space. Let $I$ be a finite non-empty set,
and let $\{C_{i}\}_{i\in I}$ be a collection of cones in $X$, not
necessarily closed or proper. Let $\Omega$ be a topological space.

If $C(\Omega,X)\neq\{0\}$ and $C(\Omega,X)=\sum_{i\in I}C(\Omega,C_{i})$,
then $X=\sum_{i\in I}C_{i}$; similar statements hold for $C_{b}(\Omega,X),$
$C_{0}(\Omega,X)$ and $C_{c}(\Omega,X)$. \end{lem}
\begin{proof}
If there exists $0\neq f\in C(\Omega,X)$, then composing $f$ with
a suitable continuous linear functional on $X$ yields a non-zero
$\varphi\in C(\Omega)$. Choose $\omega_{0}\in\Omega$ such that $\varphi(\omega_{0})\neq0$;
we may assume that $\varphi(\omega_{0})=1$. If $x\in X$, then (employing
the usual notation) $\varphi\otimes x=\sum_{i\in I}f_{i}$, for some
$f_{i}\in C(\Omega,C_{i})\,\,(i\in I)$ by assumption. Specializing
this to the point $\omega_{0}$ shows that $x=\sum_{i\in I}c_{i},$
for some $c_{i}\in C_{i}\,\,(i\in I).$ 

The proofs for $C_{b}(\Omega,X)$ , $C_{0}(\Omega,X)$ and $C_{c}(\Omega,X)$
are similar.\end{proof}
\begin{cor}
\label{cor:equivalence_decomposition_for_X_and_for_function_spaces}Let
$X$ be a real or complex Banach space. Let $I$ be a non-empty finite
set, and let $\{C_{i}\}_{i\in I}$ be a collection of closed not necessarily
proper cones in $X$. Let $\Omega$ be a topological space. If $C_{c}(\Omega)\neq\{0\},$
for example if $\Omega$ is a non-empty locally compact Hausdorff
space, then the following are equivalent:\end{cor}
\begin{enumerate}
\item $X=\sum_{i\in I}C_{i}$;
\item $C(\Omega,X)=\sum_{i\in I}C(\Omega,C_{i})$;
\item $C_{b}(\Omega,X)=\sum_{i\in I}C_{b}(\Omega,C_{i})$;
\item $C_{0}(\Omega,X)=\sum_{i\in I}C_{0}(\Omega,C_{i})$;
\item $C_{c}(\Omega,X)=\sum_{i\in I}C_{c}(\Omega,C_{i})$.\end{enumerate}
\begin{proof}
If $X=\{0\}$ there is nothing to prove. If $X\neq\{0\}$, then the
fact that $C_{c}(\Omega)\neq\{0\}$ implies that $C_{c}(\Omega,X)\neq\{0\}$,
hence that the other three spaces of $X$-valued functions are non-zero
as well. Combining Theorem \ref{thm:continuous_functions_arbitrary_topological_space}
and Lemma \ref{lem:decomposition_for_X_from_that_for_functions} therefore
concludes the proof.
\end{proof}
Theorem \ref{thm:continuous_functions_arbitrary_topological_space}
and Corollary \ref{cor:equivalence_decomposition_for_X_and_for_function_spaces}
are based on part (1) of Theorem \ref{thm:continuous_selections_from_generating_cones_ell_one_case}.
It is also possible to take part (2) into account and, e.g., obtain
results on various types of conormality for spaces of continuous functions
with values in a (pre)-ordered Banach space. Here is an example, where
part (2) of Theorem \ref{thm:continuous_selections_from_generating_cones_ell_one_case}
is used via an appeal to Corollary \ref{cor:continous_approximate_conormality}.
Note that, analogous to Corollary \ref{cor:equivalence_decomposition_for_X_and_for_function_spaces},
the approximate $\alpha$-conormality of $X$ and of the three normed
spaces of $X$-valued functions are all equivalent if $C_{c}(\Omega)\neq\{0\}$.
\begin{cor}
\label{cor:equivalence_approximate_normality_for_X_and_for_function_spaces}Let
$X$ be a real \textup{(}pre\textup{)}-ordered Banach space, \textup{(}pre\textup{)}-ordered
by a closed generating not necessarily proper cone $X^{+}$, and let
$\Omega$ be a topological space. Suppose that $\alpha>0$. 

If $X$ is approximately $\alpha$-conormal, then so are $C_{b}(\Omega,X$),
$C_{0}(\Omega,X)$, and $C_{c}(\Omega,X)$.

If $C_{b}(\Omega,X)\neq\{0\}$ and $C_{b}(\Omega,X)$ is approximately
$\alpha$-conormal, then $X$ is approximately $\alpha$-conormal. 

If $C_{0}(\Omega,X)\neq\{0\}$ and $C_{0}(\Omega,X)$ is approximately
$\alpha$-conormal, then $X$ is approximately $\alpha$-conormal. 

If $C_{c}(\Omega,X)\neq\{0\}$ and $C_{c}(\Omega,X)$ is approximately
$\alpha$-conormal, then $X$ is approximately $\alpha$-conormal.\end{cor}
\begin{proof}
Let $X$ be approximately $\alpha$-conormal, and let $f\in C_{b}(\Omega,X)$
and $\varepsilon>0$ be given. Corollary \ref{cor:continous_approximate_conormality}
supplies continuous positively homogeneous maps $\gamma_{\varepsilon}^{\pm}:X\to X^{+}$
, such that $x=\gamma_{\varepsilon}^{+}(x)-\gamma_{\varepsilon}^{-}(x)$
and $\Vert\gamma_{\varepsilon}^{+}(x)\Vert\leq(\alpha+\varepsilon)\Vert x\Vert$,
for all $x\in X$. Then $\gamma_{\varepsilon}^{\pm}\circ f\in C_{b}(\Omega,X^{+})$,
$f=\gamma_{\varepsilon}^{+}\circ f-\gamma_{\varepsilon}^{-}\circ f$,
and $\Vert\gamma_{\varepsilon}^{\pm}\circ f\Vert_{\infty}\leq(\alpha+\varepsilon)\Vert f\Vert_{\infty}$,
so that $C_{b}(\Omega,X)$ is approximately $\alpha$-conormal. The
proof for $C_{0}(\Omega,X)$ and $C_{c}(\Omega,X)$ is similar. 

If $C_{b}(\Omega,X)\neq\{0\}$ and $C_{b}(\Omega,X)$ is approximately
$\alpha$-conormal, let $x\in X$ and $\varepsilon>0$ be given. As
in the proof of Corollary \ref{lem:decomposition_for_X_from_that_for_functions}
we find a non-zero $\varphi\in C_{b}(\Omega)$, and we may assume
that $\Vert\varphi\Vert_{\infty}=1$ and $\varphi$ is real-valued.
Passing to $-\varphi$ if necessary we obtain a sequence $\{\omega_{n}\}\subset\Omega$
such that $0<\varphi(\omega_{n})\uparrow1$. Hence there exists $\omega_{n_{0}}\in\Omega$
such that $0<\varphi(\omega_{n_{0}})$ and $(\alpha+\varepsilon/2)\varphi(\omega_{n_{0}})^{-1}<\alpha+\varepsilon$.
By assumption, there exist $f^{\pm}\in C_{b}(\Omega,X^{+})$, such
that $\varphi\otimes x=f^{+}-f^{-}$ and $\Vert f^{+}\Vert_{\infty}\leq(\alpha+\varepsilon/2)\Vert\varphi\otimes x\Vert_{\infty}=(\alpha+\varepsilon/2)\Vert x\Vert$.
In particular, $\varphi(\omega_{n_{0}})x=f^{+}(\omega_{n_{0}})-f^{-}(\omega_{n_{0}})$.
Since $\varphi(\omega_{n_{0}})^{-1}f^{\pm}(\omega_{n_{0}})\in X^{+}$,
and $\Vert\varphi(\omega_{n_{0}})^{-1}f^{+}(\omega_{n_{0}})\Vert\leq\varphi(\omega_{n_{0}})^{-1}\Vert f^{+}\Vert_{\infty}\leq\varphi(\omega_{n_{0}})^{-1}(\alpha+\varepsilon/2)\Vert x\Vert\leq(\alpha+\varepsilon)\Vert x\Vert$,
we conclude that $X$ is approximately $\alpha$-conormal.

The proofs for $C_{0}(\Omega,X)$ and $C_{c}(\Omega,X)$ are similar.\end{proof}
\begin{rem}
In the context of Corollary \ref{cor:equivalence_approximate_normality_for_X_and_for_function_spaces},
the conclusion that the Banach spaces $C_{b}(\Omega,X)$ and $C_{0}(\Omega,X)$
are approximately $\alpha$-conormal shows that part (1) of Corollary
\ref{cor:continous_approximate_conormality} is satisfied for these
(pre)-ordered Banach spaces. Hence (2) is valid as well. Therefore,
if $X$ is approximately $\alpha$-conormal, then $C_{b}(\Omega,X)$
and $C_{0}(\Omega,X)$ are continuously positively homogeneously approximately
$\alpha$-conormal in the sense of part (2) of Corollary \ref{cor:continous_approximate_conormality}.
The converse holds for $C_{b}(\Omega,X)$ if this space is non-zero,
and similarly for $C_{0}(\Omega,X)$. 
\end{rem}

\subsection*{Acknowledgements}

The authors would like to thank Klaas Pieter Hart for helpful correspondence
and Anthony Wickstead for bringing the references \citep{AsimowAtkinson}
and \citep{Wickstead} to their attention.

	% % % % % % % % % % % % % % % % %
	%	Bibliography
	% % % % % % % % % % % % % % % % %
	\bibliographystyle{amsplain}
	\bibliography{bib/all.bib}

\providecommand{\bysame}{\leavevmode\hbox to3em{\hrulefill}\thinspace}
\providecommand{\MR}{\relax\ifhmode\unskip\space\fi MR }
\providecommand{\MRhref}[2]{%
  \href{http://www.ams.org/mathscinet-getitem?mr=#1}{#2}
}
\providecommand{\href}[2]{#2}
\begin{thebibliography}{10}

\bibitem{AliprantisBorder}
C.D. Aliprantis and K.C. Border, \emph{Infinite dimensional analysis}, third
  ed., Springer, Berlin, 2006.

\bibitem{Ando}
T.~And{\^o}, \emph{On fundamental properties of a {B}anach space with a cone},
  Pacific J. Math. \textbf{12} (1962), 1163--1169.

\bibitem{AsimowAtkinson}
L.~Asimow and H.~Atkinson, \emph{Dominated extensions of continuous affine
  functions with range an ordered {B}anach space}, Quart. J. Math. Oxford Ser.
  (2) \textbf{23} (1972), 383--389.

\bibitem{BattyRobinson}
C.J.K. Batty and D.W. Robinson, \emph{Positive one-parameter semigroups on
  ordered {B}anach spaces}, Acta Appl. Math. \textbf{2} (1984), no.~3-4,
  221--296.

\bibitem{Klee}
V.L. Klee, Jr., \emph{Boundedness and continuity of linear functionals}, Duke
  Math. J. \textbf{22} (1955), 263--269.

\bibitem{KreinNormality}
M.~Krein, \emph{Propri\'et\'es fondamentales des ensembles coniques normaux
  dans l'espace de {B}anach}, C. R. (Doklady) Acad. Sci. URSS (N.S.)
  \textbf{28} (1940), 13--17.

\bibitem{NgOpenMapping}
K.F. Ng, \emph{An open mapping theorem}, Proc. Cambridge Philos. Soc.
  \textbf{74} (1973), 61--66.

\bibitem{Rudin}
W.~Rudin, \emph{Functional analysis}, second ed., McGraw-Hill Inc., New York,
  1991.

\bibitem{Stone}
A.H. Stone, \emph{Paracompactness and product spaces}, Bull. Amer. Math. Soc.
  \textbf{54} (1948), 977--982.

\bibitem{Valero}
O.~Valero, \emph{Closed graph and open mapping theorems for normed cones},
  Proc. Indian Acad. Sci. Math. Sci. \textbf{118} (2008), no.~2, 245--254.

\bibitem{Walsh}
B.~Walsh, \emph{Ordered vector sequence spaces and related classes of linear
  operators}, Math. Ann. \textbf{206} (1973), 89--138.

\bibitem{Wickstead}
A.W. Wickstead, \emph{Spaces of linear operators between partially ordered
  {B}anach spaces}, Proc. London Math. Soc. (3) \textbf{28} (1974), 141--158.

\end{thebibliography}

\end{document}